\documentclass[preprint,12pt]{elsarticle}
\usepackage{amsthm,amsmath,amssymb}
\usepackage{graphicx}
\usepackage[colorlinks=true,citecolor=black,linkcolor=black,urlcolor=blue]{hyperref}
\usepackage{mathrsfs}

\theoremstyle{plain}
\newtheorem{theorem}{Theorem}[section]
\newtheorem{lemma}[theorem]{Lemma}

\theoremstyle{definition}
\newtheorem{definition}[theorem]{Definition}
\newtheorem{example}[theorem]{Example}

\theoremstyle{remark}
\newtheorem{remark}[theorem]{Remark}

\allowdisplaybreaks

\begin{document}

\title{A direct proof of  well-definedness for the polymatroid Tutte polynomial}

\author{Xiaxia Guan$^{a,b}$~~~Xian'an Jin$^{a,}$\footnote{Corresponding author.} \\
\small $^a$School of Mathematical Sciences, Xiamen University, P. R. China\\
\small $^b$Department of Mathematics, Taiyuan University of Technology, P. R. China\\
\small \emph{Email addresses}: gxx0544@126.com; xajin@xmu.edu.cn}

\date{}
\begin{abstract}
For a polymatroid $P$ over $[n]$, Bernardi, K\'{a}lm\'{a}n and Postnikov [\emph{Adv. Math.} 402 (2022) 108355] introduced the polymatroid Tutte polynomial $\mathscr{T}_{P}$ relying on the order $1<2<\cdots<n$ of $[n]$, which generalizes the classical Tutte polynomial from matroids to polymatroids. They proved the independence of this order  by the fact that $\mathscr{T}_{P}$ is equivalent to another polynomial that only depends on $P$. In this paper, similar to the Tutte's original proof of the well-definedness of the Tutte polynomial defined by the summation over all spanning trees using activities depending on the order of edges, we give a direct and elementary proof of the well-definedness of the polymatroid Tutte polynomial.
\end{abstract}

\maketitle

\section{Introduction}
\noindent

In 1954,  Tutte \cite{Tutte} introduced the Tutte polynomial of graphs as follows.
\begin{definition}
Let $G$ be a graph and let $\tau$ be a spanning tree of $G$. Given an ordering of edges of $G$, an edge $e$ in   $\tau$ is called  \emph{internally active} if $(\tau\setminus e)\cup e'$ is not a spanning tree of $G$ for all $e'<e$. Let $i_{G}(\tau)$ denote the number of internally edges with respect
to the spanning tree $\tau$. An edge $e$ in   $E(G)\setminus\tau$ is called  \emph{externally active} if $(\tau\cup e)\setminus e'$ is not a spanning tree of $G$ for all $e'<e$. Let $j_{G}(\tau)$ denote the number of externally edges with respect
to the spanning tree $\tau$. The \emph{Tutte polynomial} $T_{G}(x,y)$ of the graph $G$ is  the following sum
over all spanning trees $\tau$ of $G$: $$T_{G}(x, y)=\sum_{\tau \ \text{is a spanning tree of}\ G }x^{i_{G}(\tau)}y^{j_{G}(\tau)}.$$
\end{definition}

The definition of the Tutte polynomial depends on the order of edges, but Tutte \cite{Tutte} proved that it is an invariant of graphs. It generalizes many classical graph invariants, such as the chromatic polynomial, the flow polynomial etc. Furthermore, the Tutte polynomial and its evalutions play an important role in statistical physics, knot theory, and many other areas of mathematics and physics.

Crapo \cite{Crapo} extended the Tutte polynomial from graphs to matroids. As a generalization of the Tutte polynomial $T_{M}$ of matroids $M$, Bernardi, K\'{a}lm\'{a}n and Postnikov \cite{Bernardi}  defined the polymatroid Tutte polynomial $\mathscr{T}_{P}$ for polymatroids $P$.
We first recall the definition of polymatroids. Throughout the paper, let  $[n]=\{1,2,\ldots,n\}$, $2^{[n]}=\{I|I\subset[n]\}$, and let $\textbf{e}_{1},\textbf{e}_{2},\ldots,\textbf{e}_{n}$ denote the canonical basis of $\mathbb{R}^{n}$.
\begin{definition} \label{def polymatroid}
A \emph{polymatroid}\footnote{In many papers, cf.~\cite{Edmonds}, polymatroids are often defined as slightly larger sets, and the set $P$ of Definition \ref{def polymatroid} is referred to as the set of integer bases of a polymatroid. Moreover, the rank function $f$ also satisfies the monotonicity (in other words, $f(A)\leq f(B)$ if $A\subset B\subset [n]$) if $P\subset \mathbb{Z}_{\geq 0}^{n}$.} $P=P_{f}\subset \mathbb{Z}^{n}$ (i.e., on the ground set $[n]$) with the rank function $f$ is  given by
$$\left\{(a_{1},\ldots,a_{n})\in \mathbb{Z}^{n}\bigg|\sum_{i\in I}a_{i}\leq f(I) \ \text{for any}\ I\subset [n]  \ \text{and}\ \sum_{i\in [n]}a_{i}=f([n]) \right\},$$
where $f:2^{[n]}\rightarrow \mathbb{Z}$ satisfies
\begin{enumerate}
\item[(1)] $f(\emptyset)=0$;

\item[(2)] $f(I)+f(J)\geq f(I\cup J)+f(I\cap J)$ for any $I,J\subset [n]$ (submodularity).
\end{enumerate}
\end{definition}

A vector $\textbf{a}\in \mathbb{Z}^{n}$ is called a \emph{basis} of $P$ if $\textbf{a}\in P$. It is easy to see that the set of bases (viewed as element of $\{0,1\}^{n}$) of any matroid is a polymatroid.

We now recall the definition of  the polymatroid Tutte polynomial.
\begin{definition} \label{pop}
Let $P$ be a polymatroid  over $[n]$.
For a basis $\textbf{a}\in P$, an index $i\in [n]$ is \emph{internally active} if $\textbf{a}-\textbf{e}_{i}+\textbf{e}_{j}\notin P$  for any $j<i$. Let $\mathrm{Int}(\textbf{a})=\mathrm{Int}_{P}(\textbf{a})\subset [n]$ denote the set of all
internally active indices with respect to $\textbf{a}$.

An index $i\in [n]$ is \emph{externally active} if $\textbf{a}+\textbf{e}_{i}-\textbf{e}_{j}\notin P$ for any $j<i$. Let $\mathrm{Ext}(\textbf{a})=\mathrm{Ext}_{P}(\textbf{a})\subset [n]$ denote the set of all externally active indices with respect to $\textbf{a}$.

The \emph{polymatroid Tutte polynomial} $\mathscr{T}_{P}(x,y)$ is defined as
$$\mathscr{T}_{P}(x,y):=\sum_{\textbf{a}\in P}x^{oi(\textbf{a})}y^{oe(\textbf{a})}(x+y-1)^{ie(\textbf{a})},$$
where $$oi(\textbf{a}):=|\mathrm{Int}(\textbf{a})\setminus \mathrm{Ext}(\textbf{a})|,$$ $$oe(\textbf{a}):=|\mathrm{Ext}(\textbf{a})\setminus \mathrm{Int}(\textbf{a})|,$$ $$ie(\textbf{a}):=|\mathrm{Int}(\textbf{a})\cap \mathrm{Ext}(\textbf{a})|.$$
\end{definition}

Bernardi et al.~\cite{Bernardi} showed that if $M\subset 2^{[n]}$ is a matroid of rank $d$ over $[n]$, and  $P=P(M)\subset \{0,1\}^{n}$ is its corresponding polymatroid, then
$$T_{M}(x,y)=\frac{(x+y-xy)^{n}}{x^{n-d}y^{d}}\mathscr{T}_{P}(\frac{x}{x+y-xy},\frac{y}{x+y-xy}).$$

\begin{definition}
For any $\textbf{a},\textbf{b}\in \mathbb{R}^{n}$, define
$$d^{>}(\textbf{a},\textbf{b}):=\sum_{i:a_{i}>b_{i}}(a_{i}-b_{i})\ \text{and} \ d^{<}(\textbf{a},\textbf{b}):=\sum_{i:b_{i}>a_{i}}(b_{i}-a_{i}).$$
\end{definition}

\begin{definition}
For any polymatroid $P\subset \mathbb{Z}^{n}$ and any vector $\textbf{c}\in \mathbb{Z}^{n}$, define $$d^{>}(P,\textbf{c}):=\min_{\textbf{a}\in P}d^{>}(\textbf{a},\textbf{c})\ \text{and} \ \
d^{<}(P,\textbf{c}):=\min_{\textbf{a}\in P}d^{<}(\textbf{a},\textbf{c}).$$
Moreover, define
the formal power series $\widetilde{\mathscr{T}}_{P}(u,v)$ in two variables $u$ and $v$ by $$\widetilde{\mathscr{T}}_{P}(u,v):=\sum_{\textbf{c}\in \mathbb{Z}^{n}}wt_{P}(\textbf{c}),$$
where $$wt_{P}(\textbf{c}):=u^{d^{>}(P,\textbf{c})}v^{d^{<}(P,\textbf{c})}.$$
\end{definition}
Bernardi et al.\ \cite{Bernardi} established a relation between $\mathscr{T}_{P}$ and $\widetilde{\mathscr{T}}_{P}$.
\begin{theorem}\cite[Theorem 10.6]{Bernardi}\label{T-T}
For any polymatroid $P$, we have that $$\widetilde{\mathscr{T}}_{P}(u,v)=\mathscr{T}_{P}(\frac{1}{1-u},\frac{1}{1-v}).$$
\end{theorem}

They also proved that the polymatroid Tutte polynomial is equivalent to the one introduced by Cameron and Fink \cite{Cameron}.
For a polymatroid $P$, although the definition of $\mathscr{T}_{P}$ relies on the natural order $1<2<\cdots<n$ of the ground set $[n]$, $\widetilde{\mathscr{T}}_{P}$ depends only on $P$. Hence, Theorem \ref{T-T} implies that $\mathscr{T}_{P}$ is an invariant of polymatroids $P$. Namely, for a permutation
$w\in S_{n}$ and a point $\textbf{a}= (a_{1},\ldots, a_{n})\in \mathbb{Z}^{n}$, define
$w(\textbf{a}): = (a_{w(1)},\ldots, a_{w(n)})$, and
$$w(P) := \{w(\textbf{a})| \textbf{a} \in P\},$$
where $S_{n}$ is the symmetric group acting on $\mathbb{Z}^{n}$ by permutations of coordinates. Then $\mathscr{T}_{P}$ is $S_{n}$-invariant, which can be stated as follows.
\begin{theorem}\cite[Theorem 4.7]{Bernardi}\label{$S_{n}$-invariance}
Let $P$ be a polymatroid on $[n]$. Then
$$\mathscr{T}_{P}(x,y)=\mathscr{T}_{w(P)}(x,y) \ \text{for any} \ w\in S_{n}.$$
\end{theorem}

In this paper, similar to the idea that Tutte \cite{Tutte} proved that his polynomial does not depend on the order of edges, our proof is direct and elementary for Theorem \ref{$S_{n}$-invariance}.
\section{Several lemmas}
\noindent

The main aim of this section is to provide several lemmas used in the next section.

\begin{lemma}\label{transmit}
Let $P$ be a polymatroid on $[n]$. For any basis $\textbf{a}\in P$ and three distinct indices $i,j,k\in [n]$, if $\textbf{a}+\textbf{e}_{k}-\textbf{e}_{i}\in P$ and $\textbf{a}+\textbf{e}_{i}-\textbf{e}_{j}\in P$, then $\textbf{a}+\textbf{e}_{k}-\textbf{e}_{j}\in P$.
\end{lemma}

\begin{proof}
Let $\textbf{b}=\textbf{a}+\textbf{e}_{k}-\textbf{e}_{i},$ $\textbf{c}=\textbf{a}+\textbf{e}_{i}-\textbf{e}_{j}$ and $\textbf{d}=\textbf{a}+\textbf{e}_{k}-\textbf{e}_{j}.$ Let $f$ be the rank function of the polymatroid $P$. It is easy to see that $$\sum_{i'\in [n]}d_{i'}=\sum_{i'\in [n]}a_{i'}=f([n]).$$
We now divide two cases to prove that $\sum\limits_{i'\in I}d_{i'}\leq f(I)$ for any subset $I\subset [n]$.
\begin{enumerate}
\item[(i)] If $i\notin I$, then  $\sum\limits_{i'\in I}d_{i'}\leq \sum\limits_{i'\in I}b_{i'}\leq f(I).$
\item[(ii)] If $i\in I$, then  $\sum\limits_{i'\in I}d_{i'}\leq \sum\limits_{i'\in I}c_{i'}\leq f(I).$
\end{enumerate}
Hence, the vector $\textbf{d}=\textbf{a}+\textbf{e}_{k}-\textbf{e}_{j}\in P$ by the definition of polymatroids.
\end{proof}

\begin{example}\label{EX}
Let $P\subset \mathbb{Z}^{5}$ be a polymatroid. The rank function $f:2^{[n]}\rightarrow \mathbb{Z}$ of the polymatroid $P$ is given by $f(\emptyset)=0$, $f(\{1\})=f(\{3\})=1$, $f(\{2\})=f(\{4\})=f(\{5\})=f(\{1,2\})=f(\{1,3\})=f(\{2,3\})=f(\{4,5\})=f(\{1,2,3\})=2$ and $f(I)=3$ for the other $I\subset \{1,2,3,4,5\}$. It is easy to see that $f$ satisfies the submodularity and the monotonicity. Then
\begin{multline*}
P=\{\textbf{a}^{1}=(1,0,0,2,0),\textbf{a}^{2}=(1,0,0,0,2),\textbf{a}^{3}=(1,1,0,1,0),\\
\textbf{a}^{4}=(1,1,0,0,1),\textbf{a}^{5}=(1,0,1,1,0),\textbf{a}^{6}=(1,0,1,0,1),\textbf{a}^{7}=(1,0,0,1,1),\\
\textbf{a}^{8}=(0,2,0,1,0),\textbf{a}^{9}=(0,2,0,0,1),\textbf{a}^{10}=(0,1,0,2,0),\textbf{a}^{11}=(0,0,1,2,0),\\
\textbf{a}^{12}=(0,1,0,0,2,),\textbf{a}^{13}=(0,0,1,0,2),\textbf{a}^{14}=(0,1,1,1,0),\\
\textbf{a}^{15}=(0,1,1,0,1),\textbf{a}^{16}=(0,1,0,1,1),\textbf{a}^{17}=(0,0,1,1,1)\}.
\end{multline*}
We have that $\textbf{a}^{10}+\textbf{e}_{2}-\textbf{e}_{4}=(0,2,0,1,0)=\textbf{a}^{8}$, $\textbf{a}^{10}-\textbf{e}_{2}+\textbf{e}_{1}=(1,0,0,2,0)=\textbf{a}^{1}$ and $\textbf{a}^{10}+\textbf{e}_{1}-\textbf{e}_{4}=(1,1,0,1,0)=\textbf{a}^{3}\in P$.
\end{example}

Let $P$ be a polymatroid on $[n]$. We now consider two distinct bases $\textbf{a},\textbf{b}\in P$ so that there are two indices $i,j\in [n]$ with $a_{t}=b_{t}$ for all $t\in [n]\setminus \{i,j\}$.
\begin{lemma}\label{transmit2}
Let $P$ be a polymatroid on $[n]$. Let $\textbf{a}$ and $\textbf{b}$ be distinct bases of $P$ for which there are two indices $i,j\in [n]$ so that $a_{i}<b_{i}$ and $a_{t}=b_{t}$ for all $t\in [n]\setminus \{i,j\}$ (which imply $a_{j}>b_{j}$ and $i\neq j$). Then for any $k\in[n]\setminus \{i,j\}$, the following conclusions hold.
\begin{itemize}
  \item [(1)] If $\textbf{a}+\textbf{e}_{k}-\textbf{e}_{j}\in P$, then $\textbf{b}+\textbf{e}_{k}-\textbf{e}_{i}\in P$.
  \item [(2)] If $\textbf{a}-\textbf{e}_{k}+\textbf{e}_{i}\in P$, then $\textbf{b}-\textbf{e}_{k}+\textbf{e}_{j}\in P$.
\end{itemize}
\end{lemma}
\begin{proof}
Let $\textbf{c}=\textbf{a}+\textbf{e}_{k}-\textbf{e}_{j},$ $\textbf{d}=\textbf{b}+\textbf{e}_{k}-\textbf{e}_{i},$ $\textbf{f}=\textbf{a}-\textbf{e}_{k}+\textbf{e}_{i}$ and $\textbf{g}=\textbf{b}-\textbf{e}_{k}+\textbf{e}_{j}.$
Let $f$ be the rank function of the polymatroid $P$. Note that $\textbf{b}\in P$, then $$\sum_{i'\in [n]}d_{i'}=\sum_{i'\in [n]}g_{i'}=\sum_{i'\in [n]}b_{i'}=f([n]).$$ We now prove that $\sum\limits_{i'\in I}d_{i'}\leq f(I)$ and $\sum\limits_{i'\in I}g_{i'}\leq f(I)$ for any subset $I\subset [n]$.
\begin{enumerate}
\item[(i)] If $i\notin I$, then  $\sum\limits_{i'\in I}d_{i'}\leq\sum\limits_{i'\in I}c_{i'}\leq f(I).$
\item[(ii)] If $i\in I$, then  $\sum\limits_{i'\in I}d_{i'}\leq\sum\limits_{i'\in I}b_{i'}\leq f(I).$
\item[(iii)] If $j\in I$, then $\sum\limits_{i'\in I}g_{i'}\leq\sum\limits_{i'\in I}f_{i'}\leq f(I).$
\item[(iv)] If $j\notin I$, then $\sum\limits_{i'\in I}g_{i'}\leq\sum\limits_{i'\in I}b_{i'}\leq f(I).$
\end{enumerate}
Therefore, $\textbf{d},\textbf{g}\in P$.
\end{proof}

\begin{example}
For Example \ref{EX}, let $\textbf{a}=\textbf{a}^{10}$ and $\textbf{b}=\textbf{a}^{11}$. Then $i=3$ and $j=2$. We also have that $\textbf{a}^{10}-\textbf{e}_{2}+\textbf{e}_{1}=(1,0,0,2,0)=\textbf{a}^{1}$, $\textbf{a}^{11}-\textbf{e}_{3}+\textbf{e}_{1}=(1,0,0,2,0)=\textbf{a}^{1}$,
$\textbf{a}^{10}+\textbf{e}_{3}+\textbf{e}_{4}=(0,1,1,1,0)=\textbf{a}^{14}$ and $\textbf{a}^{11}+\textbf{e}_{2}-\textbf{e}_{4}=(0,1,1,1,0)=\textbf{a}^{14}\in P$.
\end{example}

Let $P$ be a polymatroid  over $[n]$ with the rank function $f$. For a basis $\textbf{a}\in P$, let
$$\mathcal{I}(\textbf{a}):=\mathcal{I}_{P}(\textbf{a})=\left\{I\subset [n]\bigg| \sum_{i\in I}a_{i}= f(I)\right\}.$$

Obviously, $\emptyset, [n]\in \mathcal{I}(\textbf{a})$ for any basis $\textbf{a}\in P$. The next result holds from \cite[Theorem 44.2]{Schrijver}, since $f$ is submodular.

\begin{lemma}\label{tight}\cite{Schrijver}
Let $P$ be a polymatroid. For any basis $\textbf{a}\in P$, if $I,J\in \mathcal{I}(\textbf{a})$, then $I\cup J,I\cap J\in \mathcal{I}(\textbf{a})$.
\end{lemma}

The following conclusion was obtained in \cite[Lemma 4.2]{Bernardi}.
\begin{lemma} \cite{Bernardi}\label{ia}
Let $P$ be a polymatroid  over $[n]$. For any $\textbf{a}\in P$,

\begin{itemize}
  \item [(1)] an index $i\in [n]$ is internally active with respect to $\textbf{a}$
if and only if there exists a subset $I\subset [n]$ such that $i=\min(I)$ and $[n]\setminus I\in \mathcal{I}(\textbf{a})$; and
  \item [(2)] an index $i\in [n]$ is externally active with respect to $\textbf{a}$
if and only if there exists a subset $I'\subset [n]$ such that $i=\min(I')$ and $I'\in \mathcal{I}(\textbf{a})$.
\end{itemize}
\end{lemma}
The following statement is from \cite[Corllary 12]{Guan3}.
\begin{lemma}\label{a and b'}\cite{Guan3}
 For a polymatroid $P$ on $[n]$, let $\textbf{a}$ and $\textbf{b}$ be distinct bases of $P$ for which there are two indices $i,j\in [n]$ so that $a_{i}<b_{i}$ and $a_{t}=b_{t}$ for all $t\in [n]\setminus \{i,j\}$. Then for any $k>\max\{i,j\}$, we have that
  $k\in \mathrm{Int}(\textbf{a})$ if and only if $k\in \mathrm{Int}(\textbf{b})$.
\end{lemma}

\begin{example}
For Example \ref{EX}, let $\textbf{a}=\textbf{a}^{3}$ and $\textbf{b}=\textbf{a}^{8}$. Then $i=2$ and $j=1$. It is easy to prove that $\mathrm{Int}(\textbf{a})\cap \{3,4,5\}=\{3,4,5\}=\mathrm{Int}(\textbf{b})\cap \{3,4,5\}$.
\end{example}

The similar result also holds for the externally active indices.

\begin{lemma}\label{a and b}
For a polymatroid $P$ on $[n]$, let $\textbf{a}$ and $\textbf{b}$ be distinct bases of $P$ for which there are  $i,j\in [n]$ with $a_{i}<b_{i}$ and $a_{t}=b_{t}$ for all $t\in [n]\setminus \{i,j\}$. Then for any $k>\max\{i,j\}$, we have that
  $k\in \mathrm{Ext}(\textbf{b})$ if and only if $k\in \mathrm{Ext}(\textbf{a})$.
\end{lemma}
\begin{proof}
It is enough to prove that for any $k>i$,
if $k\in \mathrm{Ext}(\textbf{b})$, then $k\in \mathrm{Ext}(\textbf{a})$. Let $f$ be the rank function of $P$.
 If $k\in \mathrm{Ext}(\textbf{b})$, then by Lemma \ref{ia} (2), there exists a  subset $I\subset [n]$ so that $k=\min(I)$ and $I\in \mathcal{I}(\textbf{b})$. Note that $$f(I)\geq\sum_{i'\in  I}a_{i'}\geq\sum_{i'\in  I}b_{i'}=f(I)$$ as $i\notin I$. Then $I\in \mathcal{I}(\textbf{a})$. Therefore, $k\in \mathrm{Ext}(\textbf{a})$ by using Lemma \ref{ia} (2) again.
\end{proof}

\begin{example}
For Example \ref{EX}, let $\textbf{a}=\textbf{a}^{10}$ and $\textbf{b}=\textbf{a}^{11}$. Then $i=3$ and $j=2$. It is easy to prove that $\mathrm{Ext}(\textbf{a})\cap \{4,5\}=\{4\}=\mathrm{Ext}(\textbf{b})\cap \{4,5\}$.
\end{example}

Now let us focus on the smaller indices.

\begin{lemma}\label{a and b2}
Let $P$ be a polymatroid on $[n]$. Let $\textbf{a}$ and $\textbf{b}$ be distinct bases of $P$ for which there are two indices $i,j\in [n]$ so that $a_{i}<b_{i}$ and $a_{t}=b_{t}$ for all $t\in [n]\setminus \{i,j\}$. Let $t=\min\{i,j\}$. Then if there is a subset $I\subset [n]\setminus\{i,j\}$ with $[t-1]\subset I$ so that $[n]\setminus I\in \mathcal{I}(\textbf{a})\cap \mathcal{I}(\textbf{b})$, then for any $k<t$,
\begin{itemize}
  \item [(1)]  $k\in \mathrm{Int}(\textbf{a})$ if and only if $k\in \mathrm{Int}(\textbf{b})$; and
  \item [(2)]  $k\in \mathrm{Ext}(\textbf{a})$ if and only if $k\in \mathrm{Ext}(\textbf{b})$.
\end{itemize}
\end{lemma}
\begin{proof}
It is enough to prove that
\begin{itemize}
  \item [($1'$)]  if $k\in \mathrm{Int}(\textbf{a})$, then $k\in \mathrm{Int}(\textbf{b})$; and
  \item [($2'$)]  if $k\in \mathrm{Ext}(\textbf{a})$, then $k\in \mathrm{Ext}(\textbf{b})$.
\end{itemize}

Let $f$ be the rank function of  $P$.

We first prove ($1'$). By Lemma  \ref{ia} (1), there exists a  subset $I'\subset [n]$ so that $k=\min(I')$ and $[n]\setminus I'\in \mathcal{I}(\textbf{a})$ as $k\in \mathrm{Int}(\textbf{a})$. Set $I''=([n]\setminus I)\cup I'$. Then by Lemma \ref{tight},  $[n]\setminus I''=I\cap([n]\setminus I')\in \mathcal{I}(\textbf{a})$. Note that $j\in  I''$ as $j\in  I$. Then $$f([n]\setminus I'')\geq \sum_{t'\in [n]\setminus I''} b_{t'}\geq\sum_{t'\in [n]\setminus I''} a_{t'}=f([n]\setminus I'').$$ We have that $[n]\setminus I''\in \mathcal{I}(\textbf{b})$. Hence, $k\in \mathrm{Int}(\textbf{b})$ as $k=\min(I'')$.

Next, we prove ($2'$). If $k\in \mathrm{Ext}(\textbf{a})$, then by Lemma  \ref{ia} (2), there exists a  subset $S'$ so that $k=\min(S')$ and $S'\in \mathcal{I}(\textbf{a})$. Then $S= I\cap S'\in \mathcal{I}(\textbf{a})$ by Lemma \ref{tight}. Note that $j\notin S$. Then $$f(S)\geq \sum_{t'\in S} b_{t'}\geq\sum_{t'\in S} a_{t'}=f(S).$$ This implies that $S\in \mathcal{I}(\textbf{b})$.
Hence, $k\in \mathrm{Ext}(\textbf{b})$ as $k=\min(S)$.
\end{proof}

\begin{example}
For Example \ref{EX}, let $\textbf{a}=\textbf{a}^{5}$ and $\textbf{b}=\textbf{a}^{6}$. Then $i=5$ and $j=4$. Note that $[n]\setminus I=\{1,2,3\}\in \mathcal{I}(\textbf{a})\cap \mathcal{I}(\textbf{b})$. It is easy to prove that $\mathrm{Int}(\textbf{a})\cap \{1,2,3\}=\{1,2\}=\mathrm{Int}(\textbf{b})\cap \{1,2,3\}$ and $\mathrm{Ext}(\textbf{a})\cap \{1,2,3\}=\{1,3\}=\mathrm{Ext}(\textbf{b})\cap \{1,2,3\}$.
\end{example}

Similarly, we can obtain the following result.
\begin{lemma}\label{a and b3}
Let $P$ be a polymatroid on $[n]$. Let $\textbf{a}$ and $\textbf{b}$ be distinct bases of $P$ for which there are  $i,j\in [n]$ so that $a_{i}<b_{i}$ and $a_{t}=b_{t}$ for all $t\in [n]\setminus \{i,j\}$. Let $t=\min\{i,j\}$. Then if there is a subset $I\subset [n]\setminus [t-1]$ with $\{i,j\}\subset I$ so that $I\in \mathcal{I}(\textbf{a})\cap \mathcal{I}(\textbf{b})$, then for any $k<t$,
\begin{itemize}
  \item [(1)]  $k\in \mathrm{Int}(\textbf{a})$ if and only if $k\in \mathrm{Int}(\textbf{b})$; and
  \item [(2)]  $k\in \mathrm{Ext}(\textbf{a})$ if and only if $k\in \mathrm{Ext}(\textbf{b})$.
\end{itemize}
\end{lemma}

\begin{example}
For Example \ref{EX}, let $\textbf{a}=\textbf{a}^{10}$ and $\textbf{b}=\textbf{a}^{12}$. Then $i=5$ and $j=4$. Note that $I=\{4,5\}\in \mathcal{I}(\textbf{a})\cap \mathcal{I}(\textbf{b})$. It is easy to prove that $\mathrm{Int}(\textbf{a})\cap \{1,2,3\}=\{1,3\}=\mathrm{Int}(\textbf{b})\cap \{1,2,3\}$ and $\mathrm{Ext}(\textbf{a})\cap \{1,2,3\}=\{1,2\}=\mathrm{Ext}(\textbf{b})\cap \{1,2,3\}$.
\end{example}

\begin{remark}
Let $P$ be a polymatroid on $[n]$. The \emph{dual} polymatroid  of $P$ is denoted by $-P$, where
$$-P:=\{(-a_{1},\ldots,-a_{n})|(a_{1},\ldots,a_{n})\in P\}.$$
It is easy to see that for any base $\textbf{a}$ and for any indies $k,k'$, if $\textbf{b}=\textbf{a}+\textbf{e}_{k}-\textbf{e}_{k'}\in P$, then $-\textbf{b}=-\textbf{a}-\textbf{e}_{k}+\textbf{e}_{k'}\in -P$.

Then in Lemma \ref{transmit2} (2), if $\textbf{a}-\textbf{e}_{k}+\textbf{e}_{i}\in P$, then $-\textbf{a}+\textbf{e}_{k}-\textbf{e}_{i}\in -P$. Note that $-a_{i}>-b_{i}$ in $-P$, then $-\textbf{b}+\textbf{e}_{k}-\textbf{e}_{j}\in -P$ by Lemma \ref{transmit2} (1), so,
$\textbf{b}-\textbf{e}_{k}+\textbf{e}_{j}\in P$.

Moreover, for any base $\textbf{a}$ and for any subset $I$, we have that $I\in \mathcal{I}_{P}(\textbf{a})$ if and only if $[n]\setminus I\in \mathcal{I}_{-P}(\textbf{-a})$.  For any index $k$, we have $k\in \mathrm{Int}_{P}(\textbf{a})$ if and if only $k\in \mathrm{Ext}_{-P}(\textbf{-a})$, and $k\in \mathrm{Ext}_{P}(\textbf{a})$ if and if only $k\in \mathrm{Int}_{-P}(\textbf{-a})$. Hence, Lemmas \ref{a and b} and \ref{a and b3} can also be proved by using duality again for Lemmas \ref{a and b'} and \ref{a and b2}, respectively.
\end{remark}
\section{Proof of Theorem \ref{$S_{n}$-invariance}}
\noindent

In this section, we give an elementary proof of Theorem \ref{$S_{n}$-invariance}.

It is enough to prove that $$\mathscr{T}_{P}(x,y)=\mathscr{T}_{w(P)}(x,y),$$ where $$w(P):=\{(a_{1},\ldots,a_{h+1},a_{h},\ldots,a_{n})\mid(a_{1},\ldots,a_{h},a_{h+1},\ldots,a_{n})\in P\}.$$ Namely, $$1<2<\cdots<h-1<h+1<h<\cdots<n \ \text{in} \ w(P).$$

Note that for any vector $\textbf{a}\in \mathbb{Z}^{n}$, we have that $\textbf{a}\in P$ if and only if $w(\textbf{a})\in w(P)$. Moreover, for two distinct indices $i,j\in [n]$, the vector $\textbf{a}+\textbf{e}_{i}-\textbf{e}_{j}\in P$ if and only if $w(\textbf{a})+w(\textbf{e}_{i})-w(\textbf{e}_{j})\in w(P)$. Hence, for simplicity, we will state $\textbf{a}+\textbf{e}_{i}-\textbf{e}_{j}\in P$ if $\textbf{a}+\textbf{e}_{i}-\textbf{e}_{j}\in P$ or $w(\textbf{a})+w(\textbf{e}_{i})-w(\textbf{e}_{j})\in w(P)$ holds.

It is easy to see the following Facts 1,  2 and 3.

\textbf{Fact 1.} For any $\textbf{a}\in P$ and for any $i\in [n]\setminus \{h,h+1\}$,
\begin{enumerate}
  \item [(i)] $i\in \mathrm{Int}_{P}(\textbf{a})$ if and only if $i\in \mathrm{Int}_{w(P)}(w(\textbf{a}))$;
  \item [(ii)] $i\in \mathrm{Ext}_{P}(\textbf{a})$ if and only if $i\in \mathrm{Ext}_{w(P)}(w(\textbf{a}))$.
\end{enumerate}

\textbf{Fact 2.} For any $\textbf{a}\in P$,
 \begin{enumerate}
  \item [(i)] if $h\notin \mathrm{Int}_{P}(\textbf{a})$, then $h\notin \mathrm{Int}_{w(P)}(w(\textbf{a}))$;
  \item [(ii)] if $h\notin \mathrm{Ext}_{P}(\textbf{a})$, then $h\notin \mathrm{Ext}_{w(P)}(w(\textbf{a}))$.
\end{enumerate}

\textbf{Fact 3.} For any $\textbf{a}\in P$,
\begin{enumerate}
  \item [(i)] if $h+1\notin \mathrm{Int}_{w(P)}(w(\textbf{a}))$, then $h+1\notin \mathrm{Int}_{P}(\textbf{a})$;
  \item [(ii)] if $h+1\notin \mathrm{Ext}_{w(P)}(w(\textbf{a}))$, then $h+1\notin \mathrm{Ext}_{P}(\textbf{a})$.
\end{enumerate}

For any $\textbf{a}\in P$,
define its  associated subset $\mathcal{F}_{\textbf{a}}\subset P$ for $\textbf{a}$ in $P$ by
\begin{eqnarray*}
\mathcal{F}_{\textbf{a}}: &=&\left\{(b_{1},\ldots,b_{n})\in P\bigg | \  b_{i}=a_{i} \
\text{for all}\ i\in [n]\setminus \{h,h+1\}\right\}\\
&=&\{\textbf{b}^{k}:k=1,2,\ldots,l\}.
\end{eqnarray*}
Define $\textbf{b}^{1}\in \mathcal{F}_{\textbf{a}}$ such that $b^{1}_{h}$ is the smallest, and $\textbf{b}^{k+1}=\textbf{b}^{k}+\textbf{e}_{h}-\textbf{e}_{h+1}$. Then there is some index $m$ such that $\textbf{b}^{m}=\textbf{a}$, and  $\textbf{b}^{l} \in \mathcal{F}_{\textbf{a}}$ such that $b^{l}_{h}$ is the largest.

Note that $P$ is the union of such pairwise disjoint associated sets $\mathcal{F}_{\textbf{a}}$. It is enough to prove
\begin{eqnarray*}
&&\sum _{\textbf{b}\in \mathcal{F}_{\textbf{a}}}x^{oi_{P}(\textbf{b})}y^{oe_{P}(\textbf{b})}(x+y-1)^{ie_{P}(\textbf{b})}\\
&=& \sum _{w(\textbf{b})\in \mathcal{F}_{w(\textbf{a})}}x^{oi_{w(P)}(w(\textbf{b}))}y^{oe_{w(P)}(w(\textbf{b}))}(x+y-1)^{ie_{w(P)}(w(\textbf{b}))}.~~~~~~~(*)
\end{eqnarray*}

For any $\textbf{a}\in P$,  let $$\textbf{a}'=\textbf{a}+\textbf{e}_{h}-\textbf{e}_{h+1}$$ and
$$\textbf{a}''=\textbf{a}+\textbf{e}_{h+1}-\textbf{e}_{h}.$$

To prove the above equation, we need the following claims.

\textbf{Claim 1.} Assume that $\textbf{a}'\in P$.

\begin{enumerate}
  \item [(i)] If $h\notin \mathrm{Int}_{P}(\textbf{a})$, then $h+1\notin \mathrm{Int}_{w(P)}(w(\textbf{a}))$.
  \item [(ii)] If $h+1\notin \mathrm{Ext}_{w(P)}(w(\textbf{a}))$, then $h\notin \mathrm{Ext}_{P}(\textbf{a})$.
\end{enumerate}

\emph{Proof of Claim 1.} (i) If $h\notin \mathrm{Int}_{P}(\textbf{a})$, then there is some $k<h$ so that $\textbf{a}+\textbf{e}_{k}-\textbf{e}_{h}\in P$. Since $\textbf{a}'=\textbf{a}+\textbf{e}_{h}-\textbf{e}_{h+1} \in P$, by Lemma \ref{transmit}, we have that $\textbf{a}+\textbf{e}_{k}-\textbf{e}_{h+1}\in P$. Hence, $h+1\notin \mathrm{Int}_{w(P)}(w(\textbf{a}))$.

(ii) If $h+1\notin \mathrm{Ext}_{w(P)}(w(\textbf{a}))$, then there is some $k'<h$ so that $\textbf{a}-\textbf{e}_{k'}+\textbf{e}_{h+1}\in P$. By Lemma \ref{transmit}, $\textbf{a}-\textbf{e}_{k'}+\textbf{e}_{h}\in P$ as $\textbf{a}' \in P$. Therefore, $h\notin \mathrm{Ext}_{P}(\textbf{a})$.

  \textbf{Claim 2.} Assume that $\textbf{a}'\notin P$.

\begin{enumerate}
  \item [(i)] If $h+1\notin \mathrm{Int}_{P}(\textbf{a})$, then $h+1\notin \mathrm{Int}_{w(P)}(w(\textbf{a}))$.
  \item [(ii)] If $h\notin \mathrm{Ext}_{w(P)}(w(\textbf{a}))$, then $h\notin \mathrm{Ext}_{P}(\textbf{a})$.
\end{enumerate}

\emph{Proof of Claim 2.} (i) If $h+1\notin \mathrm{Int}_{P}(\textbf{a})$, then there is some index $k<h+1$ so that $\textbf{a}+\textbf{e}_{k}-\textbf{e}_{h+1}\in P$. Note that $k\neq h$ as $\textbf{a}'=\textbf{a}+\textbf{e}_{h}-\textbf{e}_{h+1}\notin P$. This implies that $k<h$. Hence, $h+1\notin \mathrm{Int}_{w(P)}(w(\textbf{a}))$.

(ii) If $h\notin \mathrm{Ext}_{w(P)}(w(\textbf{a}))$, then there is some index $k'<h$ or $k'=h+1$ so that  $\textbf{a}-\textbf{e}_{k'}+\textbf{e}_{h}\in P$. Note also that $k'<h$ as $\textbf{a}'=\textbf{a}+\textbf{e}_{h}-\textbf{e}_{h+1}\notin P$. This implies that $h\notin \mathrm{Ext}_{P}(\textbf{a})$.

We have Claims 3 and 4 similar to Claims 1 and 2.

\textbf{Claim 3.} Assume that $\textbf{a}''\in P$.
\begin{enumerate}
  \item [(i)] If $h+1\notin \mathrm{Int}_{w(P)}(w(\textbf{a}))$, then $h\notin \mathrm{Int}_{P}(\textbf{a})$.
  \item [(ii)] If $h\notin \mathrm{Ext}_{P}(\textbf{a})$, then $h+1\notin \mathrm{Ext}_{w(P)}(w(\textbf{a}))$.
\end{enumerate}

 \textbf{Claim 4.} Assume that $\textbf{a}''\notin P$.

\begin{enumerate}
  \item [(i)] If $h\notin \mathrm{Int}_{w(P)}(w(\textbf{a}))$, then $h\notin \mathrm{Int}_{P}(\textbf{a})$.
  \item [(ii)] If $h+1\notin \mathrm{Ext}_{P}(\textbf{a})$, then $h+1\notin \mathrm{Ext}_{w(P)}(w(\textbf{a}))$.
\end{enumerate}

We now consider the following two cases for $\mathcal{F}_{\textbf{a}}$.

\textbf{Case 1.} Assume that $|\mathcal{F}_{\textbf{a}}|=1$, that is, $\mathcal{F}_{\textbf{a}}=\{\textbf{a}\}$.

In this case,  both $\textbf{a}'\notin P$ and $\textbf{a}''\notin P$ hold. Then by Facts 2 and 3, and Claims 2 and 4, we have
\begin{enumerate}
  \item [(1-i)] $h\in \mathrm{Int}_{P}(\textbf{a})$ if and only if $h\in \mathrm{Int}_{w(P)}(w(\textbf{a}))$ (by Claim 4 (i) and Fact 2 (i));
  \item [(1-ii)] $h+1\in \mathrm{Int}_{P}(\textbf{a})$ if and only if $h+1\in \mathrm{Int}_{w(P)}(w(\textbf{a}))$ (by Fact 3 (i) and Claim 2 (i));
   \item [(1-iii)] $h\in \mathrm{Ext}_{P}(\textbf{a})$ if and only if $h\in \mathrm{Ext}_{w(P)}(w(\textbf{a}))$ (by Fact 2 (ii) and Claim 2 (ii));
  \item [(1-iv)] $h+1\in \mathrm{Ext}_{P}(\textbf{a})$ if and only if $h+1\in \mathrm{Ext}_{w(P)}(w(\textbf{a}))$ (by Fact 3 (ii) and Claim 4 (ii)).
\end{enumerate}

Combining with Fact 1, for any $i\in [n]$, we have that
$i\in \mathrm{Int}_{P}(\textbf{a})$ if and only if $i\in \mathrm{Int}_{w(P)}(w(\textbf{a}))$, and $i\in \mathrm{Ext}_{P}(\textbf{a})$ if and only if $i\in \mathrm{Ext}_{w(P)}(w(\textbf{a}))$. Hence,
 $$x^{oi_{P}(\textbf{a})}y^{oe_{P}(\textbf{a})}(x+y-1)^{ie_{P}(\textbf{a})}=
x^{oi_{w(P)}(w(\textbf{a}))}y^{oe_{w(P)}(w(\textbf{a}))}(x+y-1)^{ie_{w(P)}(w(\textbf{a}))}.$$

Hence, the equation (*) holds since $\mathcal{F}_{\textbf{a}}=\{\textbf{a}\}$.

\begin{example}
For Example \ref{EX} and $h=3$, let $\textbf{a}^{3}=(1,1,0,1,0)$. We know that $\textbf{a}^{3}-\textbf{e}_{3}+\textbf{e}_{4}=(1,1,-1,2,0)\notin P$ and $\textbf{a}^{3}-\textbf{e}_{4}+\textbf{e}_{3}=(1,1,1,0,0)\notin P$. So, in this case, $\mathcal{F}_{\textbf{a}^{3}}=\{\textbf{a}^{3}\}$. In  both the order $1<2<3<4<5$ and the order $1<2<4<3<5$, we have that $\mathrm{Int}_{P}(\textbf{a}^{3})=\mathrm{Int}_{w(P)}(w(\textbf{a}^{3}))=\{1,2,3,4,5\}$ and $\mathrm{Ext}_{P}(\textbf{a}^{3})=\mathrm{Ext}_{w(P)}(w(\textbf{a}^{3}))=\{1\}$.
\end{example}

\textbf{Case 2.} Assume that $|\mathcal{F}_{\textbf{a}}|\geq 2$.

We firstly prove $$x^{oi_{P}(\textbf{a})}y^{oe_{P}(\textbf{a})}(x+y-1)^{ie_{P}(\textbf{a})}=
x^{oi_{w(P)}(w(\textbf{a}))}y^{oe_{w(P)}(w(\textbf{a}))}(x+y-1)^{ie_{w(P)}(w(\textbf{a}))}$$ for any $\textbf{a}\in \mathcal{F}_{\textbf{a}} \setminus \{\textbf{b}^{1},\textbf{b}^{l}\}$. In this case, we have $\textbf{a}'\in P$ and $\textbf{a}''\in P$. Then
\begin{enumerate}
  \item [(2-i)] $h\notin \mathrm{Int}_{w(P)}(w(\textbf{a}))\cup \mathrm{Ext}_{w(P)}(w(\textbf{a}))$ and $h+1\notin \mathrm{Int}_{P}(\textbf{a})\cup \mathrm{Ext}_{P}(\textbf{a})$;
\end{enumerate}
by  Claims 1 and 3, we have
\begin{enumerate}
  \item [(2-ii)] $h\in \mathrm{Int}_{P}(\textbf{a})$ if and only if $h+1\in \mathrm{Int}_{w(P)}(w(\textbf{a}))$ (by Claims 1 (i) and 3 (i));
   \item [(2-iii)] $h\in \mathrm{Ext}_{P}(\textbf{a})$ if and only if $h+1\in \mathrm{Ext}_{w(P)}(w(\textbf{a}))$ (by Claims 1 (ii) and 3 (ii)).
\end{enumerate}
Combining with Fact 1, we also have $$x^{oi_{P}(\textbf{a})}y^{oe_{P}(\textbf{a})}(x+y-1)^{ie_{P}(\textbf{a})}=
x^{oi_{w(P)}(w(\textbf{a}))}y^{oe_{w(P)}(w(\textbf{a}))}(x+y-1)^{ie_{w(P)}(w(\textbf{a}))}.$$

Next let $$A=x^{oi_{P}(\textbf{b}^{1})}y^{oe_{P}(\textbf{b}^{1})}(x+y-1)^{ie_{P}(\textbf{b}^{1})},$$
$$B=x^{oi_{P}(\textbf{b}^{l})}y^{oe_{P}(\textbf{b}^{l})}(x+y-1)^{ie_{P}(\textbf{b}^{l})},$$
$$C=x^{oi_{w(P)}(w(\textbf{b}^{1}))}y^{oe_{w(P)}(w(\textbf{b}^{1}))}(x+y-1)^{ie_{w(P)}(w(\textbf{b}^{1}))},$$ and
$$D=x^{oi_{w(P)}(w(\textbf{b}^{l}))}y^{oe_{w(P)}(w(\textbf{b}^{l}))}(x+y-1)^{ie_{w(P)}(w(\textbf{b}^{l}))}.$$
To prove the equation (*) is enough to prove that $$A+B=C+D.$$

\begin{example}
For Example \ref{EX} and $h=4$, let $\textbf{a}^{7}=(1,0,0,1,1)$. We know that $\textbf{a}^{7}-\textbf{e}_{4}+\textbf{e}_{5}=(1,0,0,0,2)=\textbf{a}^{2}\in P$ and $\textbf{a}^{7}-\textbf{e}_{5}+\textbf{e}_{4}=(1,0,0,2,0)=\textbf{a}^{1}\in P$. In this case, $\mathcal{F}_{\textbf{a}^{7}}=\{\textbf{a}^{1},\textbf{a}^{2},\textbf{a}^{7}\}$. In  the order $1<2<3<4<5$, we have that $\mathrm{Int}_{P}(\textbf{a}^{7})=\{1,2,3\}$ and $\mathrm{Ext}_{P}(\textbf{a}^{7})=\{1,4\}$.
In the order $1<2<3<5<4$, we have that $\mathrm{Int}_{w(P)}(w(\textbf{a}^{7}))=\{1,2,3\}$ and $\mathrm{Ext}_{w(P)}(w(\textbf{a}^{7}))=\{1,5\}$. It is enough to prove that
\begin{eqnarray*}
&&\sum _{\textbf{b}\in \{\textbf{a}^{1},\textbf{a}^{2}\}}x^{oi_{P}(\textbf{b})}y^{oe_{P}(\textbf{b})}(x+y-1)^{ie_{P}(\textbf{b})}\\
&=& \sum _{w(\textbf{b})\in \{\textbf{a}^{1},\textbf{a}^{2}\}}x^{oi_{w(P)}(w(\textbf{b}))}y^{oe_{w(P)}(w(\textbf{b}))}(x+y-1)^{ie_{w(P)}(w(\textbf{b}))}.
\end{eqnarray*}
\end{example}

By the definitions of $\textbf{b}^{1}$ and $\textbf{b}^{l}$, we have that $\textbf{b}^{1}+\textbf{e}_{h}-\textbf{e}_{h+1}\in P$, $\textbf{b}^{1}+\textbf{e}_{h+1}-\textbf{e}_{h}\notin P$, $\textbf{b}^{l}+\textbf{e}_{h+1}-\textbf{e}_{h}\in P$ and $\textbf{b}^{l}+\textbf{e}_{h}-\textbf{e}_{h+1}\notin P$. Then
\begin{enumerate}
  \item [(3-i)] $h+1\notin \mathrm{Int}_{P}(\textbf{b}^{1})$ and $h\notin \mathrm{Ext}_{w(P)}(w(\textbf{b}^{1}))$ (due to $\textbf{b}^{1}+\textbf{e}_{h}-\textbf{e}_{h+1}\in P$);

  \item [(3-ii)] $h\in \mathrm{Int}_{P}(\textbf{b}^{1})$ if and only if $h\in \mathrm{Int}_{w(P)}(w(\textbf{b}^{1}))$ (by Fact 2 (i) and Claim 4 (i) as $\textbf{b}^{1}+\textbf{e}_{h+1}-\textbf{e}_{h}\notin P$);
   \item [(3-iii)] $h+1\in \mathrm{Ext}_{P}(\textbf{b}^{1})$ if and only if $h+1\in \mathrm{Ext}_{w(P)}(w(\textbf{b}^{1}))$ (by Fact 3 (ii) and Claim 4 (ii) as $\textbf{b}^{1}+\textbf{e}_{h+1}-\textbf{e}_{h}\notin P$);
   \item [(3-iv)] $h\notin \mathrm{Int}_{w(P)}(w(\textbf{b}^{l}))$ and $h+1\notin \mathrm{Ext}_{P}(\textbf{b}^{l})$ (due to $\textbf{b}^{l}+\textbf{e}_{h+1}-\textbf{e}_{h}\in P$);
   \item [(3-v)] $h+1\in \mathrm{Int}_{P}(\textbf{b}^{l})$ if and only if  $h+1\in \mathrm{Int}_{w(P)}(w(\textbf{b}^{l}))$  (by Fact 3 (i) and Claim 2 (i) as $\textbf{b}^{l}+\textbf{e}_{h}-\textbf{e}_{h+1}\notin P$);
   \item [(3-vi)] $h\in \mathrm{Ext}_{P}(\textbf{b}^{l})$ if and only if  $h\in \mathrm{Ext}_{w(P)}(w(\textbf{b}^{l}))$ (by Fact 2 (ii) and Claim 2 (ii) as $\textbf{b}^{l}+\textbf{e}_{h}-\textbf{e}_{h+1}\notin P$).
\end{enumerate}

Combining with Fact 1,
\begin{enumerate}
  \item [($3'$-i)] for any $i\in [n]\setminus \{h+1\}$, we have that $i\in \mathrm{Int}_{P}(\textbf{b}^{1})$ if and only if $i\in \mathrm{Int}_{w(P)}(w(\textbf{b}^{1}))$,  and $i\in \mathrm{Ext}_{P}(\textbf{b}^{l})$ if and only if  $i\in \mathrm{Ext}_{w(P)}(w(\textbf{b}^{l}))$;

  \item [($3'$-ii)] for any $j\in [n]\setminus \{h\}$, we have that $j\in \mathrm{Ext}_{P}(\textbf{b}^{1})$ if and only if  $j\in \mathrm{Ext}_{w(P)}(w(\textbf{b}^{1}))$,  and $j\in \mathrm{Int}_{P}(\textbf{b}^{l})$ if and only if  $j\in \mathrm{Int}_{w(P)}(w(\textbf{b}^{l}))$.
\end{enumerate}

We need the following Claims 5-8.

\textbf{Claim 5.} If $h+1\notin \mathrm{Int}_{w(P)}(w(\textbf{b}^{1}))$, then $h\notin \mathrm{Int}_{P}(\textbf{b}^{l})$.

\emph{Proof of Claim 5.}  If $h+1\notin \mathrm{Int}_{w(P)}(w(\textbf{b}^{1}))$, then there exists some $k<h$ so that $\textbf{b}^{1}+\textbf{e}_{k}-\textbf{e}_{h+1}\in P$ (Note that $w(\textbf{b}^{1})+\textbf{e}_{k}-w(\textbf{e}_{h+1})\in w(P)$ if and only if $\textbf{b}^{1}+\textbf{e}_{k}-\textbf{e}_{h+1}\in P$). By Lemma \ref{transmit2} (1),  $\textbf{b}^{l}+\textbf{e}_{k}-\textbf{e}_{h}\in P$ (set $h+1=j$, $h=i$, $\textbf{b}^{1}=\textbf{a}$ and $\textbf{b}^{l}=\textbf{b}$). Hence, $h\notin \mathrm{Int}_{P}(\textbf{b}^{l})$.

\textbf{Claim 6.} Assume that $h+1\in \mathrm{Int}_{w(P)}(w(\textbf{b}^{1}))$. Then

\begin{enumerate}
  \item [(i)] $h+1\in \mathrm{Int}_{w(P)}(w(\textbf{b}^{l}))$ and  $h\in \mathrm{Int}_{P}(\textbf{b}^{1})\cap \mathrm{Int}_{P}(\textbf{b}^{l})$;
  \item [(ii)] for any $i\in [n]\setminus\{h,h+1\}$, we have that $i\in \mathrm{Int}_{P}(\textbf{b}^{1})$ if and only if $i\in \mathrm{Int}_{P}(\textbf{b}^{l})$, and $i\in \mathrm{Ext}_{P}(\textbf{b}^{1})$ if and only if $i\in \mathrm{Ext}_{P}(\textbf{b}^{l})$.
\end{enumerate}

\emph{Proof of Claim 6.} (i) By Lemma \ref{ia} (1), there is a subset $I\in \mathcal{I}(\textbf{b}^{1})$ so that  $[h-1]\subset I\subset[n]\setminus\{h+1\}$ as $h+1\in \mathrm{Int}_{w(P)}(w(\textbf{b}^{1}))$. Note that $$f(I)=\sum_{t\in I} b^{1}_{t}\leq \sum_{t\in I} b^{l}_{t}\leq f(I).$$ Then $h\notin I$ and $I\in \mathcal{I}(\textbf{b}^{l})$. Hence, $I\in \mathcal{I}(w(\textbf{b}^{l}))$, which implies $h+1\in \mathrm{Int}_{w(P)}(w(\textbf{b}^{l}))$ and  $h\in \mathrm{Int}_{P}(\textbf{b}^{1})\cap \mathrm{Int}_{P}(\textbf{b}^{l})$.

(ii) We know that $[h-1]\subset I\subset[n]\setminus\{h,h+1\}$ and $I\in \mathcal{I}(\textbf{b}^{1})\cap \mathcal{I}(\textbf{b}^{l})$ by the proof of Claim 6 (i). Then the conclusion holds for any $i<h$ by Lemma \ref{a and b2} (set $h=i$, $h+1=j$, $I=I$, $\textbf{b}^{1}=\textbf{a}$ and $\textbf{b}^{l}=\textbf{b}$). It follows from Lemmas  \ref{a and b'} and \ref{a and b} for any $i>h+1$.

\textbf{Claim 7.} If $h\notin \mathrm{Ext}_{P}(\textbf{b}^{1})$, then $h+1\notin \mathrm{Ext}_{w(P)}(w(\textbf{b}^{l}))$.

\emph{Proof of Claim 7.}   If $h\notin \mathrm{Ext}_{P}(\textbf{b}^{1})$, then there exists some $k<h$ so that $\textbf{b}^{1}-\textbf{e}_{k}+\textbf{e}_{h}\in P$. By Lemma \ref{transmit2} (2), $\textbf{b}^{l}-\textbf{e}_{k}+\textbf{e}_{h+1}\in P$ (Set $h+1=j$, $h=i$, $\textbf{b}^{1}=\textbf{a}$ and $\textbf{b}^{l}=\textbf{b}$). Hence, $h+1\notin \mathrm{Ext}_{w(P)}(w(\textbf{b}^{l}))$.

\textbf{Claim 8.} Assume that $h\in\mathrm{Ext}_{P}(\textbf{b}^{1})$. Then

\begin{enumerate}
  \item [(i)] $h+1\in \mathrm{Ext}_{w(P)}(w(\textbf{b}^{1}))\cap \mathrm{Ext}_{w(P)}(w(\textbf{b}^{l}))$ and $h\in \mathrm{Ext}_{P}(\textbf{b}^{l})$,
  \item [(ii)] for any $i\in [n]\setminus\{h,h+1\}$, we have that $i\in \mathrm{Int}_{P}(\textbf{b}^{1})$ if and only if $i\in \mathrm{Int}_{P}(\textbf{b}^{l})$, and $i\in \mathrm{Ext}_{P}(\textbf{b}^{1})$ if and only if $i\in \mathrm{Ext}_{P}(\textbf{b}^{l})$.
\end{enumerate}

\emph{Proof of Claim 8.} (i) Since $h\in \mathrm{Ext}_{P}(\textbf{b}^{1})$, by Lemma  \ref{ia} (2), there is a subset $I\subset[n]$ so that $h=\min(I)$ and $I\in \mathcal{I}(\textbf{b}^{1})$. Then $h+1\in I\in \mathcal{I}(\textbf{b}^{l})$ as $f(I)=\sum_{t\in I}b^{1}_{t}\leq \sum_{t\in I}b^{l}_{t}\leq f(I)$. Hence, $h+1\in \mathrm{Ext}_{w(P)}(w(\textbf{b}^{1}))\cap \mathrm{Ext}_{w(P)}(w(\textbf{b}^{l}))$ and $h\in \mathrm{Ext}_{P}(\textbf{b}^{l})$.

(ii) We know that $\{h,h+1\}\subset I\subset[n]\setminus [h-1]$ and $I\in \mathcal{I}(\textbf{b}^{1})\cap \mathcal{I}(\textbf{b}^{l})$ by the proof of Claim 8 (i). Then the conclusion holds for any $i<h$ by Lemma \ref{a and b3} (set $h=i$, $h+1=j$, $I=I$, $\textbf{b}^{1}=\textbf{a}$ and $\textbf{b}^{l}=\textbf{b}$). It follows from Lemmas  \ref{a and b'} and \ref{a and b} for any $i>h+1$.

We now divide Case 2 into four subcases.

\textbf{Subcase 2.1.} Assume that $h+1\notin \mathrm{Int}_{w(P)}(w(\textbf{b}^{1}))$ and $h\notin \mathrm{Ext}_{P}(\textbf{b}^{1})$. By Claims 5 and 7, $h\notin \mathrm{Int}_{P}(\textbf{b}^{l})$ and $h+1\notin \mathrm{Ext}_{w(P)}(w(\textbf{b}^{l}))$.
By (3-i),(3-iv),($3'$-i) and ($3'$-ii), for any $i\in [n]$, $i\in \mathrm{Int}_{P}(\textbf{b}^{1})$ if and only if $i\in \mathrm{Int}_{w(P)}(w(\textbf{b}^{1}))$, $i\in \mathrm{Ext}_{P}(\textbf{b}^{1})$ if and only if $i\in \mathrm{Ext}_{w(P)}(w(\textbf{b}^{1}))$, $i\in \mathrm{Int}_{P}(\textbf{b}^{l})$ if and only if $i\in \mathrm{Int}_{w(P)}(w(\textbf{b}^{l}))$, and $i\in \mathrm{Ext}_{P}(\textbf{b}^{l})$ if and only if $i\in \mathrm{Ext}_{w(P)}(w(\textbf{b}^{l}))$, that is, $A=C$ and $B=D$. Hence, $$A+B=C+D.$$

Assume that Subcase 2.1 is not true. Then by Fact 1, Claims 6 and 8, for any $i\in [n]\setminus\{h,h+1\}$, $$i\in \mathrm{Int}_{w(P)}(w(\textbf{b}^{1}))\stackrel{\text{Fact 1}}{\Longleftrightarrow} i\in \mathrm{Int}_{P}(\textbf{b}^{1})  \stackrel{\text{Claims 6 and 8}}{\Longleftrightarrow} i\in \mathrm{Int}_{P}(\textbf{b}^{l}) \stackrel{\text{Fact 1}}{\Longleftrightarrow} i\in \mathrm{Int}_{w(P)}(w(\textbf{b}^{l})),$$ and $$i\in \mathrm{Ext}_{w(P)}(w(\textbf{b}^{1})) \stackrel{\text{Fact 1}}{\Longleftrightarrow} i\in \mathrm{Ext}_{P}(\textbf{b}^{1}) \stackrel{\text{Claims 6 and 8}}{\Longleftrightarrow} i\in \mathrm{Ext}_{P}(\textbf{b}^{l}) \stackrel{\text{Fact 1}}{\Longleftrightarrow} i\in \mathrm{Ext}_{w(P)}(w(\textbf{b}^{l})).$$
This implies that the set $[n]\setminus\{h,h+1\}$ has same distribution for polynomials $A$, $B$, $C$ and $D$. We may assume that this distribution is $E$. Hence, it is enough to prove that $$\frac{A}{E}+\frac{B}{E}=\frac{C}{E}+\frac{D}{E}.$$
Namely, we only consider distributions of $h+1$ and $h$ for polynomials $A$, $B$, $C$ and $D$.

\textbf{Subcase 2.2.} Assume that $h+1\in \mathrm{Int}_{w(P)}(w(\textbf{b}^{1}))$ and $h\notin \mathrm{Ext}_{P}(\textbf{b}^{1})$. By Claims 6 (i) and 7, we have that $h+1\in \mathrm{Int}_{w(P)}(w(\textbf{b}^{l}))$,  $h\in \mathrm{Int}_{P}(\textbf{b}^{1})\cap \mathrm{Int}_{P}(\textbf{b}^{l})$ and $h+1\notin \mathrm{Ext}_{w(P)}(w(\textbf{b}^{l}))$.
By (3-ii) and (3-v), $h\in \mathrm{Int}_{w(P)}(w(\textbf{b}^{1}))$ and $h+1\in \mathrm{Int}_{P}(\textbf{b}^{l})$. We now consider the following four cases by (3-iii) and (3-vi).

\begin{enumerate}
  \item [(i)] Assume that $h\notin \mathrm{Ext}_{P}(\textbf{b}^{l})$ and $(h+1)\notin \mathrm{Ext}_{P}(\textbf{b}^{1})$. Then $h\notin \mathrm{Ext}_{w(P)}(w(\textbf{b}^{l}))$ and $h+1\notin \mathrm{Ext}_{w(P)}(w(\textbf{b}^{1}))$. Hence $$\frac{A}{E}+\frac{B}{E}=x+x^{2}=\frac{C}{E}+\frac{D}{E}.$$
\item [(ii)] Assume that $h\in \mathrm{Ext}_{P}(\textbf{b}^{l})$ and $h+1\notin \mathrm{Ext}_{P}(\textbf{b}^{1})$. Then
    $$\frac{A}{E}+\frac{B}{E}=x+x(x+y-1)=xy+x^{2}=\frac{C}{E}+\frac{D}{E}.$$
 \item [(iii)] Assume that $h\notin \mathrm{Ext}_{P}(\textbf{b}^{l})$ and $h+1\in \mathrm{Ext}_{P}(\textbf{b}^{1})$. Then
     $$\frac{A}{E}+\frac{B}{E}=xy+x^{2}=x+x(x+y-1)=\frac{C}{E}+\frac{D}{E}.$$
  \item [(iv)] Assume that $h\in \mathrm{Ext}_{P}(\textbf{b}^{l})$ and $h+1\in \mathrm{Ext}_{P}(\textbf{b}^{1})$. Then
      $$\frac{A}{E}+\frac{B}{E}=xy+x(x+y-1)=\frac{C}{E}+\frac{D}{E}.$$
\end{enumerate}
Subcase 2.2 is summarized in Table 1.
\begin{table}[htp]
\begin{center}
\renewcommand\arraystretch{1.5}
\caption{Activities of $h$ and $h+1$ in Subcase 2.2, where $IE$, $I\overline{E}$, $\overline{I}E$ and $\overline{I}\overline{E}$ denote elements of sets
$\mathrm{Int}(\textbf{b}^{1})\cap \mathrm{Ext}(\textbf{b}^{1})$, $\mathrm{Int}(\textbf{b}^{1})\setminus \mathrm{Ext}(\textbf{b}^{1})$, $\mathrm{Ext}(\textbf{b}^{1})\setminus \mathrm{Int}(\textbf{b}^{1})$ and $[n]\setminus (\mathrm{Int}(\textbf{b}^{1})\cup \mathrm{Ext}(\textbf{b}^{1}))$, respectively.}
 \begin{tabular}{|c |c|c|c|c|c|c|c|c|c|c |c|c|c|c|}
  \hline
  &(i)& (ii)& (iii)& (iv)\\
 \hline
 $(h,\textbf{b}^{1},P)$&$I\overline{E}$& $I\overline{E}$& $I\overline{E}$ & $I\overline{E}$\\
 \hline
$(h,w(\textbf{b}^{l}),w(P))$&$\overline{I}\overline{E}$& $\overline{I}E$& $\overline{I}\overline{E}$& $\overline{I}E$\\
 \hline
 $(h,w(\textbf{b}^{1}),w(P))$&$I\overline{E}$& $I\overline{E}$& $I\overline{E}$& $I\overline{E}$\\
 \hline
 $(h,\textbf{b}^{l},P)$&$I\overline{E}$& $IE$& $I\overline{E}$& $IE$\\
 \hline
 $(h+1,\textbf{b}^{1},P)$&$\overline{I}\overline{E}$& $\overline{I}\overline{E}$& $\overline{I}E$ &$\overline{I}E$\\
 \hline
$(h+1,w(\textbf{b}^{l}),w(P))$&$I\overline{E}$& $I\overline{E}$& $I\overline{E}$& $I\overline{E}$\\
 \hline
 $(h+1,w(\textbf{b}^{1}),w(P))$&$I\overline{E}$& $I\overline{E}$& $IE$& $IE$\\
 \hline
 $(h+1,\textbf{b}^{l},P)$&$I\overline{E}$& $I\overline{E}$& $I\overline{E}$&$ I\overline{E}$\\
 \hline
 $\frac{A}{E}$&$x$& $x$& $xy$& $xy$\\
 \hline
 $\frac{B}{E}$&$x^{2}$& $x(x+y-1)$&$x^{2}$& $x(x+y-1)$\\
 \hline
 $\frac{C}{E}$&$x^{2}$& $x^{2}$& $x(x+y-1)$& $x(x+y-1)$\\
 \hline
 $\frac{D}{E}$&$x$& $xy$&$x$& $xy$\\
 \hline
 $\frac{A}{E}+\frac{B}{E}$&$x+x^{2}$& $x+x(x+y-1)$& $xy+x^{2}$& $xy+x(x+y-1)$\\
 \hline
 $\frac{C}{E}+\frac{D}{E}$&$x+x^{2}$& $xy+x^{2}$&$x+x(x+y-1)$& $xy+x(x+y-1)$\\
 \hline
\end{tabular}
\end{center}
\end{table}

\textbf{Subcase 2.3.} Assume that $h+1\notin \mathrm{Int}_{w(P)}(w(\textbf{b}^{1}))$ and $h\in \mathrm{Ext}_{P}(\textbf{b}^{1})$. By Claims 8 (i) and 5, we have that $(h+1)\in \mathrm{Ext}_{w(P)}(w(\textbf{b}^{1}))\cap \mathrm{Ext}_{w(P)}(w(\textbf{b}^{l}))$,  $h\in  \mathrm{Ext}_{P}(\textbf{b}^{l})$ and $h\notin \mathrm{Int}_{P}(\textbf{b}^{l})$. By (3-iii) and (3-vi), $h+1\in \mathrm{Ext}_{P}(\textbf{b}^{1})$ and $h\in  \mathrm{Ext}_{w(P)}(w(\textbf{b}^{l}))$. By (3-ii) and (3-v), it is enough to consider the following four cases.

\begin{enumerate}
  \item [(i)] Assume that $h\notin \mathrm{Int}_{P}(\textbf{b}^{1})$ and $h+1\notin \mathrm{Int}_{P}(\textbf{b}^{l})$. Then
      $$\frac{A}{E}+\frac{B}{E}=y+y^{2}=\frac{C}{E}+\frac{D}{E}.$$
  \item [(ii)] Assume that $h\in \mathrm{Int}_{P}(\textbf{b}^{1})$ and $h+1\notin \mathrm{Int}_{P}(\textbf{b}^{l})$. Then
      $$\frac{A}{E}+\frac{B}{E}=y(x+y-1)+y=y^{2}+xy=\frac{C}{E}+\frac{D}{E}.$$
 \item [(iii)] Assume that $h\notin \mathrm{Int}_{P}(\textbf{b}^{1})$ and $h+1\in \mathrm{Int}_{P}(\textbf{b}^{l})$. Then
     $$\frac{A}{E}+\frac{B}{E}=y^{2}+xy=y(x+y-1)+y=\frac{C}{E}+\frac{D}{E}.$$
  \item [(iv)] Assume that $h\in \mathrm{Int}_{P}(\textbf{b}^{1})$ and $h+1\in \mathrm{Int}_{P}(\textbf{b}^{l})$. Then
      $$\frac{A}{E}+\frac{B}{E}=y(x+y-1)+xy=\frac{C}{E}+\frac{D}{E}.$$
\end{enumerate}
Subcase 2.3 is summarized in Table 2.
\begin{table}[htp]
\begin{center}
\renewcommand\arraystretch{1.5}
\caption{Activities of $h$ and $h+1$ in Subcase 2.3, where $IE$, $I\overline{E}$, $\overline{I}E$ and $\overline{I}\overline{E}$ denote elements of sets
$\mathrm{Int}(\textbf{b}^{1})\cap \mathrm{Ext}(\textbf{b}^{1})$, $\mathrm{Int}(\textbf{b}^{1})\setminus \mathrm{Ext}(\textbf{b}^{1})$, $\mathrm{Ext}(\textbf{b}^{1})\setminus \mathrm{Int}(\textbf{b}^{1})$ and $[n]\setminus (\mathrm{Int}(\textbf{b}^{1})\cup \mathrm{Ext}(\textbf{b}^{1}))$, respectively.}
 \begin{tabular}{|c |c|c|c|c|c|c|c|c|c|c |c|c|c|c|}
  \hline
  &(i)& (ii)& (iii)& (iv)\\
 \hline
 $(h,\textbf{b}^{1},P)$&$\overline{I}E$& $IE$&$ \overline{I}E$ & $IE$\\
 \hline
$(h,w(\textbf{b}^{l}),w(P))$&$\overline{I}E$& $\overline{I}E$& $\overline{I}E$& $\overline{I}E$\\
 \hline
 $(h,\textbf{b}^{1},w(P))$&$\overline{I}\overline{E}$& $I\overline{E}$& $\overline{I}\overline{E}$& $I\overline{E}$\\
 \hline
 $(h,\textbf{b}^{l},P)$&$\overline{I}E$& $\overline{I}E$&$\overline{I}E$& $\overline{I}E$\\
 \hline
 $(h+1,\textbf{b}^{1},P)$&$\overline{I}E$& $\overline{I}E$& $\overline{I}E $& $\overline{I}E$\\
 \hline
$(h+1,w(\textbf{b}^{l}),w(P))$&$\overline{I}E$& $\overline{I}E$& $IE$& $IE$\\
 \hline
 $(h+1,w(\textbf{b}^{1}),w(P))$&$\overline{I}E$& $\overline{I}E$& $\overline{I}E$& $\overline{I}E$\\
 \hline
 $(h+1,\textbf{b}^{l},P)$&$\overline{I}\overline{E}$& $\overline{I}\overline{E}$& $I\overline{E}$& $I\overline{E}$\\
 \hline
 $\frac{A}{E}$&$y^{2}$& $y(x+y-1)$& $y^{2}$& $y(x+y-1)$\\
 \hline
 $\frac{B}{E}$&$y$& $y$&$xy$& $xy$\\
 \hline
 $\frac{C}{E}$&$y$& $xy$& $y$& $xy$\\
 \hline
 $\frac{D}{E}$&$y^{2}$& $y^{2}$&$y(x+y-1)$& $y(x+y-1)$\\
 \hline
 $\frac{A}{E}+\frac{B}{E}$&$y+y^{2}$& $y(x+y-1)+y$& $y^{2}+xy$& $y(x+y-1)+xy$\\
 \hline
$\frac{A}{E}+\frac{B}{E}$&$y+y^{2}$ &$y^{2}+xy$& $y(x+y-1)+y$& $y(x+y-1)+xy$\\
 \hline
\end{tabular}
\end{center}
\end{table}

\textbf{Subcase 2.4.} Assume that $h+1\in \mathrm{Int}_{w(P)}(w(\textbf{b}^{1}))$ and $h\in \mathrm{Ext}_{P}(\textbf{b}^{1})$. Then by Claims 6 (i) and 8 (i), we have that $(h+1)\in \mathrm{Int}_{w(P)}(w(\textbf{b}^{l}))$,  $h\in \mathrm{Int}_{P}(\textbf{b}^{1})\cap \mathrm{Int}_{P}(\textbf{b}^{l})$, $h+1\in \mathrm{Ext}_{w(P)}(w(\textbf{b}^{1}))\cap \mathrm{Ext}_{w(P)}(w(\textbf{b}^{l}))$ and $h\in  \mathrm{Ext}_{P}(\textbf{b}^{l})$. Moreover, by (3-ii), (3-iii), (3-v) and (3-vi), we have that $h\in \mathrm{Int}_{w(P)}(w(\textbf{b}^{1}))$, $h+1\in \mathrm{Int}_{P}(\textbf{b}^{l})$, $h+1\in \mathrm{Ext}_{P}(\textbf{b}^{1})$ and $h\in  \mathrm{Ext}_{w(P)}(w(\textbf{b}^{l}))$. Then
      $$\frac{A}{E}+\frac{B}{E}=x(x+y-1)+y(x+y-1)=\frac{C}{E}+\frac{D}{E}.$$

Subcase 2.4 is summarized in Table 3.

\begin{table}[htp]
\begin{center}
\renewcommand\arraystretch{1.5}
\caption{Activities of $h$ and $h+1$ in Subcase 2.4, where $IE$, $I\overline{E}$, $\overline{I}E$ and $\overline{I}\overline{E}$ denote elements of sets
$\mathrm{Int}(\textbf{b}^{1})\cap \mathrm{Ext}(\textbf{b}^{1})$, $\mathrm{Int}(\textbf{b}^{1})\setminus \mathrm{Ext}(\textbf{b}^{1})$, $\mathrm{Ext}(\textbf{b}^{1})\setminus \mathrm{Int}(\textbf{b}^{1})$ and $[n]\setminus (\mathrm{Int}(\textbf{b}^{1})\cup \mathrm{Ext}(\textbf{b}^{1}))$, respectively.}
 \begin{tabular}{|c |c|c|c|c|c|c|c|c|c|c |c|c|c|c|}
  \hline
  $(h,\textbf{b}^{1},P)$&$(h,w(\textbf{b}^{l}),w(P))$& $(h,w(\textbf{b}^{1}),w(P))$& $(h,\textbf{b}^{l},P)$\\
 \hline
 $IE$&$\overline{I}E$& $I\overline{E}$ & $IE$\\
 \hline
 $(h+1,\textbf{b}^{1},P)$&$(h+1,\textbf{b}^{l},w(P))$& $(h+1,\textbf{b}^{1},w(P))$& $(h+1,\textbf{b}^{l},P)$\\
 \hline
$\overline{I}E$&$IE$& $IE$& $I\overline{E}$\\
 \hline
\end{tabular}
\end{center}
\end{table}

So, $$\sum _{\textbf{b}\in \{\textbf{b}^{1},\textbf{b}^{l}\}}x^{oi_{P}(\textbf{b})}y^{oe_{P}(\textbf{b})}(x+y-1)^{ie_{P}(\textbf{b})}=\sum _{\textbf{b}\in \{w(\textbf{b}^{1}),w(\textbf{b}^{l})\}}x^{oi_{w(P)}(\textbf{b})}y^{oe_{w(P)}(\textbf{b})}(x+y-1)^{ie_{w(P)}(\textbf{b})}.$$

Hence, the equation (*) holds.

In conclusion,
 $$\mathscr{T}_{P}(x,y)=\mathscr{T}_{w(P)}(x,y).$$

\begin{remark}
The proof of Theorem \ref{$S_{n}$-invariance} contains the proof of Theorem 1 in \cite{Guan4} as a special case.
\end{remark}

\begin{example}\label{EX2}
For Example \ref{EX} and $h=3$, we have that
\begin{multline*}
w(P)=\{w(\textbf{a}^{1})=(1,0,2,0,0),w(\textbf{a}^{2})=(1,0,0,0,2),w(\textbf{a}^{3})=(1,1,1,0,0),\\
w(\textbf{a}^{4})=(1,1,0,0,1),w(\textbf{a}^{5})=(1,0,1,1,0),w(\textbf{a}^{6})=(1,0,0,1,1),w(\textbf{a}^{7})=(1,0,1,0,1),\\
w(\textbf{a}^{8})=(0,2,1,0,0),w(\textbf{a}^{9})=(0,2,0,0,1),w(\textbf{a}^{10})=(0,1,2,0,0),w(\textbf{a}^{11})=(0,0,2,1,0),\\
w(\textbf{a}^{12})=(0,1,0,0,2,),w(\textbf{a}^{13})=(0,0,0,1,2),w(\textbf{a}^{14})=(0,1,1,1,0),\\
w(\textbf{a}^{15})=(0,1,0,1,1),w(\textbf{a}^{16})=(0,1,1,0,1),w(\textbf{a}^{17})=(0,0,1,1,1)\}
\end{multline*}
and $P=\{\textbf{a}^{1},\textbf{a}^{5}\}\cup \{\textbf{a}^{2}\}\cup \{\textbf{a}^{3}\}\cup \{\textbf{a}^{4}\}\cup \{\textbf{a}^{6},\textbf{a}^{7}\}\cup \{\textbf{a}^{8}\}\cup \{\textbf{a}^{9}\}\cup \{\textbf{a}^{10},\textbf{a}^{14}\}\cup \{\textbf{a}^{11}\}\cup \{\textbf{a}^{12}\}\cup \{\textbf{a}^{13}\}\cup \{\textbf{a}^{15},\textbf{a}^{16}\}\cup \{\textbf{a}^{17}\}.$ Note that $1\in \mathrm{Int}(\textbf{a})\cap \mathrm{Ext}(\textbf{a})$ for any basis $\textbf{a}\in P$. Then by the table 4, we have that
$\mathscr{T}_{P}(x,y)=(x+y-1)(x^3y+x^2y(x+y-1)+x^4+x^3+x^3y+x^2y+x^2y+x^3y+x^2y+x^2y^2+xy^2(x+y-1)+xy^2(x+y-1)+y^2(x+y-1)^2+x^2y^2+xy^2+xy^2+y^2(x+y-1))=x^5+5x^4y+10x^3y^2+10x^2y^3
+5xy^4+y^5-x^3y-5x^2y^2-6xy^3-2y^4-x^3-2x^2y+xy^2+y^3=\mathscr{T}_{w(P)}(x,y).$

\begin{table}[htp]
\begin{center}
\renewcommand\arraystretch{1.5}
\caption{Activities of $2$, $3$, $4$ and $5$ in Example \ref{EX2}, where $IE$, $I\overline{E}$, $\overline{I}E$ and $\overline{I}\overline{E}$ denote elements of sets
$\mathrm{Int}(\textbf{a})\cap \mathrm{Ext}(\textbf{a})$, $\mathrm{Int}(\textbf{a})\setminus \mathrm{Ext}(\textbf{a})$, $\mathrm{Ext}(\textbf{a})\setminus \mathrm{Int}(\textbf{a})$ and $[n]\setminus (\mathrm{Int}(\textbf{a})\cup \mathrm{Ext}(\textbf{a}))$, respectively.}
 \begin{tabular}{|c |c|c|c|c|c|c|c|c|c|c |c|c|c|c|}
  \hline
  &2& 3/w(3)& 4/w(4)& 5\\
 \hline
 $\textbf{a}^{1}/w(\textbf{a}^{1})$&$I\overline{E}$/$I\overline{E}$& $I\overline{E}$/$I\overline{E}$&$\overline{I}E$/$\overline{I}E$& $I\overline{E}$/$I\overline{E}$\\
 \hline
$\textbf{a}^{2}/w(\textbf{a}^{2})$&$I\overline{E}$/$I\overline{E}$& $I\overline{E}$/$I\overline{E}$& $IE$/$IE$& $\overline{I}E$/$\overline{I}E$\\
 \hline
 $\textbf{a}^{3}/w(\textbf{a}^{3})$&$I\overline{E}$/$I\overline{E}$& $I\overline{E}$/$I\overline{E}$& $I\overline{E}$/$I\overline{E}$& $I\overline{E}$/$I\overline{E}$\\
 \hline
 $\textbf{a}^{4}/w(\textbf{a}^{4})$& $I\overline{E}$/$I\overline{E}$&$I\overline{E}$/$I\overline{E}$& $I\overline{E}$/$I\overline{E}$& $\overline{IE}$/$\overline{IE}$\\
 \hline
 $\textbf{a}^{5}/w(\textbf{a}^{5})$& $I\overline{E}$/$I\overline{E}$& $\overline{I}E$/$\overline{I}E$& $I\overline{E}$/$I\overline{E}$& $I\overline{E}$/$I\overline{E}$\\
 \hline
 $\textbf{a}^{6}/w(\textbf{a}^{6})$&$I\overline{E}$/$I\overline{E}$& $\overline{I}E$/$\overline{I}E$&$I\overline{E}$/$I\overline{E}$ & $\overline{IE}$/$\overline{IE}$\\
 \hline
$\textbf{a}^{7}/w(\textbf{a}^{7})$&$I\overline{E}$/$I\overline{E}$& $I\overline{E}$/$I\overline{E}$& $\overline{I}E$/$\overline{I}E$& $\overline{IE}$/$\overline{IE}$\\
 \hline
 $\textbf{a}^{8}/w(\textbf{a}^{8})$&$\overline{I}E$/$\overline{I}E$& $I\overline{E}$/$I\overline{E}$& $I\overline{E}$/$I\overline{E}$& $I\overline{E}$/$I\overline{E}$\\
 \hline
 $\textbf{a}^{9}/w(\textbf{a}^{9})$& $\overline{I}E$/$\overline{I}E$&$I\overline{E}$/$I\overline{E}$& $I\overline{E}$/$I\overline{E}$& $\overline{IE}$/$\overline{IE}$\\
 \hline
 $\textbf{a}^{10}/w(\textbf{a}^{10})$& $\overline{I}E$/$\overline{I}E$& $I\overline{E}$/$I\overline{E}$& $\overline{I}E$/$\overline{I}E$& $I\overline{E}$/$I\overline{E}$\\
  \hline
 $\textbf{a}^{11}/w(\textbf{a}^{11})$&$IE$/$IE$& $\overline{I}E$/$\overline{I}E$&$ \overline{I}E$ /$\overline{I}E$& $I\overline{E}$/$I\overline{E}$\\
 \hline
$\textbf{a}^{12}/w(\textbf{a}^{12})$&$\overline{I}E$/$\overline{I}E$& $I\overline{E}$/$I\overline{E}$& $IE$/$IE$& $\overline{I}E$/$\overline{I}E$\\
 \hline
 $\textbf{a}^{13}/w(\textbf{a}^{13})$&$IE$/$IE$& $\overline{I}E$/$\overline{I}E$& $IE$/$IE$& $\overline{I}E$/$\overline{I}E$\\
 \hline
 $\textbf{a}^{14}/w(\textbf{a}^{14})$& $\overline{I}E$/$\overline{I}E$&$\overline{I}E$/$\overline{I}E$& $I\overline{E}$/$I\overline{E}$& $I\overline{E}$/$I\overline{E}$\\
 \hline
 $\textbf{a}^{15}/w(\textbf{a}^{15})$& $\overline{I}E$/$\overline{I}E$& $\overline{I}E$/$\overline{I}E$& $I\overline{E}$/$I\overline{E}$& $\overline{IE}$/$\overline{IE}$\\
 \hline
 $\textbf{a}^{16}/w(\textbf{a}^{16})$&$\overline{I}E$/$\overline{I}E$& $I\overline{E}$/$I\overline{E}$&$ \overline{I}E$/$\overline{I}E$ & $\overline{IE}$/$\overline{IE}$\\
 \hline
$\textbf{a}^{17}/w(\textbf{a}^{17})$&$IE$/$IE$ &$\overline{I}E$/$\overline{I}E$& $\overline{I}E$/$\overline{I}E$& $\overline{IE}$/$\overline{IE}$\\
 \hline
\end{tabular}
\end{center}
\end{table}

\end{example}
\section*{Acknowledgements}
\noindent
This work is supported by NSFC (No. 12171402).

\end{document}